\documentclass[12pt,oneside]{amsart}

\usepackage{amsthm,amsmath,amssymb}
\usepackage{mathrsfs}
\usepackage{enumerate}
\usepackage{amsfonts}
\usepackage{verbatim}
\usepackage{amsbsy}
\usepackage{amsmath}
\usepackage{amssymb}
\usepackage{MnSymbol}
\usepackage{mathrsfs}
\usepackage{xcolor}

\newtheorem{theorem}{Theorem}[section]
\newtheorem{lemma}[theorem]{Lemma}
\newtheorem{proposition}[theorem]{Proposition}

\newtheorem{corollary}[theorem]{Corollary} 
\theoremstyle{definition}
\newtheorem{definition}[theorem]{Definition}

\newtheorem*{claim}{Claim}
\theoremstyle{remark}
\newtheorem{remark}[theorem]{Remark}
\newtheorem{example}[theorem]{Example}

\DeclareMathOperator{\FR}{FR}

\DeclareMathOperator{\FS}{FS}

\DeclareMathOperator{\MP}{MP}

\keywords{Ramsey algebra, Hindman's Theorem, Ramsey space, orderly composition}
\subjclass[2000]{Primary 05A17; Secondary 05D10}

\begin{document}

\title{Ramsey Algebras}
\author{Wen Chean Teh}
\address{School of Mathematical Sciences\\
Universiti Sains Malaysia\\
11800 USM, Malaysia}
\email{dasmenteh@usm.my}

\begin{abstract}
Hindman's theorem says that every finite coloring of the natural numbers has a monochromatic set of finite sums. 
Ramsey algebras are  structures that satisfy an analogue of Hindman's theorem. 
This paper introduces Ramsey algebras and presents some elementary results. 
Furthemore, their connection to Ramsey spaces will be addressed.
\end{abstract}

\maketitle
\section{Introduction}

The set of natural numbers $\{0,1,2,\dotsc\}$ is denoted by $\omega$. Suppose $\langle x_i \rangle_{i \in \omega}$ is a sequence of natural numbers. Let $\FS(\langle x_i \rangle_{i \in \omega} )$ denote the set  $\{\, \sum_{i \in F} x_i \mid F \text{ is a finite nonempty subset of } \omega\,\}$. Hindman's Theorem \cite{nH74} says that
for every finite partition of the set of positive integers $\mathbb{N}=X_1\cupdot X_2\cupdot \dotsb \cupdot X_N$, there exists a sequence $\langle x_i \rangle_{i \in \omega} $ of positive natural numbers such that $\FS(\langle x_i \rangle_{i \in \omega} )\subseteq X_j$ for some $1\leq j \leq N$. 
In fact, such a sequence $\langle x_i \rangle_{i \in \omega} $ can be chosen to be a ``sum subsystem" of any given sequence $\langle y_i \rangle_{i \in \omega} $ of natural numbers.


In this paper, a class of algebraic structures called Ramsey algebras will be introduced\footnote{Ramsey algebra is also introduced by this author concurrently in \cite{wcT13a} using the notations and terminology commonly used in mathematical logic.}, each of which possesses the property analogous to that possesed by the semigroup $(\mathbb{N}, +)$ as in Hindman's Theorem. Some basic results of the theory of Ramsey algebras will be presented. In particular, a characterization of finite Ramsey algebras, as well as that of infinite Ramsey algebras involving only unary operations, is obtained.  For the remaining of this section, we give a historical account and motivation for Ramsey algebras. 


In 1988 Carlson \cite{tC88} presented an abstract version of Ellentuck's Theorem~\cite{eE74}.
Structures that have properties analogous to those of the Ellentuck space were introduced as Ramsey spaces.
Certain spaces of infinite sequences of multivariable words, each endowed with an analogous Ellentuck topology, were studied.
The single variable type of such spaces over a finite alphabet was found to be Ramsey (Theorem~2 in \cite{tC88}). This result, called Carlson's Theorem in \cite[XVIII,~\S4]{HS98}, has as corollaries many earlier Ramsey theoretic results including Hindman's Theorem, Ellentuck's Theorem, the dual Ellentuck Theorem~\cite{CS84}, the Galvin-Prikry Theorem \cite{GP73} and the Hales-Jewett Theorem \cite{HJ63}. Since then there has been an active study on Ramsey spaces (see \cite{sT10}). 

Spaces studied by Carlson are induced by some algebras, for example, algebras of variable words. His abstract Ellentuck's Theorem gives a sufficient condition for such a space to be Ramsey in terms of some combinatorial properties that boil down to the Ramsey property of the underlying algebra. Ramsey algebras can thus be thought as the combinatorial counterpart of such Ramsey spaces. 
This connection was observed by Carlson and the notion of Ramsey algebra was proposed by him.
This relation between  Ramsey algebras and Ramsey spaces will be addressed in section~6.

Hindman's Theorem implies that every semigroup is a Ramsey algebra. However, those interesting Ramsey algebras of variable words are not semigroups. The collection of operations in each of these algebras is finite but can be arbitrarily large depending on the size of the underlying finite alphabet. In view of no infinite integral domain---involving two associative binary operations---is a Ramsey algebra (Theorem~\ref{0814}), the nice interplay among the operations in a Ramsey algebra of variable words seems to be mysterious. A key feature in Carlson's proof is the construction of certain ultrafilters ``idempotent" for every operation in the corresponding algebra of variable words. These ultrafilters in turn allow the construction of ``homogeneous" sequences, establishing that the algebra is Ramsey. This approach generalizes Galvin-Glazer proof (see \cite{wC77} or \cite{nH79}) of Hindman's Theorem.

\section{Preliminaries}

The set of integers is denoted by $\mathbb{Z}$.

To us an algebra is a pair $(A, \mathcal{F})$, where $A$ is a nonempty set and $\mathcal{F}$ is a (possibly empty) collection of operations on $A$, none of which is nullary. 
If $\mathcal{F}$ is finite, we will write $(A, f_1, \dotsc, f_n)$ instead of $(A, \{f_1, \dotsc,f_n\})$. 
To say that  $a\in A$  is an \emph{idempotent} element for the algebra  $(A,\mathcal{F})$ means that $f(a,\dotsc,a)=a$ for every $f \in \mathcal{F}$.
If $B$ is a nonempty subset of $A$ and $B$ is closed under $f$ in the usual sense  for each $f\in \mathcal{F}$, then the algebra $(B, \{\,f\!\upharpoonright \! B \mid  f \in \mathcal{F}\,\})$ is called a subalgebra\footnote{Note that $f \!\upharpoonright \! B$ and $g \!\upharpoonright \! B$ can be equal when $f$ and $g$ are distinct. Our notions of ``algebra'' and ``subalgebra'' are different from but compatible with that in the theory of universal algebra.}  of $(A, \mathcal{F})$, where $f\!\upharpoonright \! B$ is the restriction of $f$ to $B^n$ with codomain $B$, provided $f$ is $n$-ary. 



The set of infinite and finite sequences  in $A$ are denoted by ${^\omega}\!A$ and ${^{<\omega}}\!A$ res\-pectively.
Suppose $\vec{a}$ is an infinite sequence $\langle a_0, a_1,a_2,\dotsc\rangle$. For every $n \geq 1$,  let $\vec{a}\!\upharpoonright \!  n$ denote the initial segment of $\vec{a}$ of length $n$ and let $\vec{a}-n$ denote the cut-off sequence $\langle a_n, a_{n+1},a_{n+2},\dotsc\rangle$. 
If $\vec{b}$ is a finite sequence $\langle b_0, b_1,\dotsc,b_{n-1}\rangle$, then $\vert \vec{b}\vert$ is its length and the concatenation $\vec{b}\ast \vec{a}$ is $\langle b_0, b_1,\dotsc,b_{n-1}, a_0, a_1,a_2,\dotsc\rangle$. Regardless of finite or infinite, $\vec{b}(i)$ is the $(i+1)$-th term of a sequence $\vec{b}$.

A preorder on a set $A$ is a binary relation on $A$ that is reflexive and transitive. 

A set in a topological space has the property of Baire if and only if its symmetric difference with some open set is meager.

If $f$ is an operation on a set $A$ and $\vec{a}$ is a finite sequence in $A$, 
whenever we write $f(\vec{a})$, it is implicitly assumed that the length of $\vec{a}$ equals the arity of $f$ and $f(\vec{a})$ stands for $f(\vec{a}(0), \dotsc,\vec{a}(\vert \vec{a}\vert -1))$.


\section{Ramsey Algebras}

A few terminologies are needed before we can introduce Ramsey algebra.
The following is a special type of composition of operations, to the author's knowledge, introduced by Carlson in  \cite{tC88}.
In fact, Carlson's definition is more general because it is defined for any heterogenoeus algebra, that is, an indexed collection of distinct nonempty sets with a collection of operations on it. 

\begin{definition}\label{0601a}
Suppose $(A,\mathcal{F})$ is an algebra, $m\in \mathbb{N}$, and $f\colon A^m \rightarrow A$. We say that $f$  is an \emph{orderly composition} of $\mathcal{F}$ if{f}
there exist $n \in \mathbb{N}$, $k_1,k_2, \dotsc, k_n\in \mathbb{N}$, and $g,h_1,h_2, \dotsc, h_n\in \mathcal{F}$ such that
\begin{enumerate}
\item $g \colon  A^n \rightarrow A$;
\item $h_i \colon  A^{k_i} \rightarrow A$ for $i=1,2,\dots,n$;
\item $m=\sum_{i=1}^n k_i$; and
\item if $\vec{x}_i\in {^{<\omega}}\!A$ and $\vert \vec{x}_i\vert=k_i$ for $i=1,2,\dotsc, n$ and $\vec{x}=\vec{x}_1\ast \vec{x}_2\ast \dotsb \ast \vec{x}_n$, then
$$f(\vec{x})=g(h_1(\vec{x}_1), h_2(\vec{x}_2), \dotsc, h_n(\vec{x}_n)).$$
\end{enumerate}
We say that $\mathcal{F}$ is \emph{closed under orderly composition} if{f} $f \in \mathcal{F}$ whenever $f$ is an orderly composition of $\mathcal{F}$. The collection of \emph{orderly terms} over $\mathcal{F}$ is  the smallest collection of operations on $A$ that includes $\mathcal{F}$, contains the identity function\footnote{In \cite{tC88} this containment of the identity function is mistakenly omitted. However, all the collections $\mathcal{F}$ encountered later in the paper do contain the identity function.} on $A$ and is closed under orderly composition.
\end{definition}


\begin{example}
Consider the integral domain $(\mathbb{Z},+,\times)$, where $+$ and $\times$ are the usual addition and multiplication. Suppose $g$ and $h$ are ternary operations on $\mathbb{Z}$ defined by $g(x,y,z)=xy+z$ and $h(x,y,z)=xz+y$ respectively. Then $g$ is an orderly term over $\{+,\times\}$ but $h$ is not.
\end{example}


Suppose $(A,\mathcal{F})$ is an algebra.
A few words will come handy on how to show that some property $P$ holds for every orderly term over $\mathcal{F}$. We can prove by induction on the generation of the orderly terms over $\mathcal{F}$. First, we show that the identity operation on $A$, as well as every operation in  $\mathcal{F}$, satisfies property $P$. 
For the induction step, assuming $h_1, \dotsc, h_n$ are orderly terms over $\mathcal{F}$ satisfying property $P$ and $g\in \mathcal{F}$ is $n$-ary, we need to show that
$f$ as defined in Definition~\ref{0601a}(4) also satisfies property $P$.


\begin{definition}\label{0524c}
Suppose $(A,\mathcal{F})$ is an algebra and $\vec{a}, \vec{b}$ are infinite sequences in $A$.
We say $\vec{b}$ is a \emph{reduction} of $\vec{a}$ with respect to $\mathcal{F}$, and write $\vec{b} \leq_{\mathcal{F}} \vec{a}$ if{f} there are finite sequences $\vec{a}_k$ and orderly terms $f_k$ over $\mathcal{F}$ for all
$k \in \omega$ such that $\vec{a}_0 \ast \vec{a}_1 \ast \vec{a}_2 \ast \dotsb$ is a subsequence of $\vec{a}$ and $\vec{b}(k)=f_k(\vec{a}_k)$ for all $k \in \omega$. 
\end{definition}

It is easy to check that $\leq_{\mathcal{F}}$ is a preorder on the collection of infinite sequences in $A$.

\begin{definition}
Suppose $(A,\mathcal{F})$ is an algebra and $\vec{b}$ is an infinite sequence in $A$. A \emph{finite reduction} of $\vec{b}$ with respect to $\mathcal{F}$ is an element of the form $f(\vec{b}_0)$ for some orderly term $f$ over  $\mathcal{F}$  and    some finite subsequence $\vec{b}_0$ of $\vec{b}$. The set of all finite reductions of $\vec{b}$ with respect to $\mathcal{F}$ is denoted by $\FR_{\mathcal{F}}(\vec{b})$.
\end{definition}

Equivalently, $\FR_{\mathcal{F}}(\vec{b})$ is equal to the set $\{\, \vec{c}(0) \mid \vec{c} \leq_{\mathcal{F}} \vec{b}\,\}$.

\begin{remark}\label{3001a}
Our definition of $\leq_{\mathcal{F}}$ is equivalent to a special case of the one given in \cite{tC88}, where the collection of operations contains all projections. 
Our choice is driven by Hindman's Theorem: $\FR_{\{+\}}(\vec{b})=\FS(\vec{b})$ for each $\vec{b} \in  {^\omega}{\mathbb{N}}$.
\end{remark}

\begin{definition}
Suppose $(A,\mathcal{F})$ is an algebra. We say that $(A, \mathcal{F})$ is a \emph{Ramsey algebra} if{f}
for every $\vec{a}\in {^\omega}\!A$ and $X \subseteq A$, there exists $\vec{b} \leq_{\mathcal{F}} \vec{a}$ such that $\vec{b}$ is \emph{homogeneous} for $X$ (with respect to $\mathcal{F}$), meaning that $\FR_{\mathcal{F}}(\vec{b})$ is either contained in or disjoint from $X$. 
\end{definition}

\begin{remark}\label{0602a}
Suppose $(A,\mathcal{F})$ is a Ramsey algebra. Since $\leq_{\mathcal{F}}$ is transitive, it can be easily deduced by induction that
 for every $\vec{a}\in {^\omega}\!A$ and finite coloring (partition) $A=X_1\cupdot X_2\cupdot \dotsb \cupdot X_N$, 
there exists $\vec{b} \leq_{\mathcal{F}} \vec{a}$ such that $\FR_{\mathcal{F}}(\vec{b} )$ is monochromatic, meaning that 
$\FR_{\mathcal{F}}(\vec{b} )\subseteq X_j$ for some $1\leq j \leq N$. 
\end{remark}

\begin{example}\label{2601a}
For any set $A$, the empty algebra $(A,\emptyset)$ is Ramsey. To see this, note that the identity function on $A$ is the only orderly term. Hence, 
$\FR_{\emptyset} (\vec{b})=\{ \,\vec{b}(i)\mid i\in \omega\,\}$ and $\vec{b}\leq_{\emptyset} \vec{a}$ simply means that $\vec{b}$ is a subsequence of $\vec{a}$. Fix $\vec{a}\in {^\omega} \!A$ and  $X\subseteq A$. Either $\{\, i\mid \vec{a}(i)\in X\,\}$ or $\{\, i\mid \vec{a}(i)\in X^c\,\}$ (or both) is infinite. Let $i_0,i_1,i_2, \dotsc$ be an increasing enumeration of the elements of whichever of these two sets that is infinite. Then
$\langle \vec{a}(i_k)\rangle_{k\in \omega}  \leq_{\emptyset}  \vec{a}$ and $\FR_{\emptyset}( \langle \vec{a}(i_k)\rangle_{k\in \omega}   )
= \{\, \vec{a}(i_k) \mid k\in \omega\,        \}$ is either contained in or disjoint from $X$ by our choice of enumeration.
\end{example}

Suppose $(A, \mathcal{F})$ is an algebra such that for every $\vec{a} \in {^\omega}\!A$, there exists  $\vec{b} \leq_{\mathcal{F}} \vec{a}$ such that  $\vert \FR_{\mathcal{F}}(\vec{b}) \vert=1$. Then $(A, \mathcal{F})$ is trivially Ramsey, and we say that it is a \emph{degenerate Ramsey algebra}. 
Note that if $ \FR_{\mathcal{F}}(\vec{b}) =\{e\}$, then $\vec{b}=  \langle e,e,e \dotsc \rangle$ and hence $e$ is an idempotent element for the algebra $(A,\mathcal{F})$.

The following characterization of finite Ramsey algebras is an unpublished observation by Carlson.  

\begin{theorem}\label{0808}
Suppose $(A, \mathcal{F})$ is a finite algebra. The following are equivalent.
\begin{enumerate}
\item $(A, \mathcal{F})$ is a Ramsey algebra.
\item For every $a\in A$, the set $\FR_{\mathcal{F}}(\langle a,a, a,\dotsc \rangle)$ contains an idempotent element for $(A,\mathcal{F})$.
\item $(A, \mathcal{F})$ is a degenerate Ramsey algebra.
\item Every subalgebra of $(A,\mathcal{F})$ contains an idempotent element for $(A,\mathcal{F})$.
\end{enumerate}
\end{theorem}

\begin{proof}
($\textit{1} \Rightarrow \textit{2}$) Fix  any element $a\in A$. Since the algebra is finite, there is a finite coloring of $A$ such that every element of $A$ gets a unique color. Since $(A,\mathcal{F})$ is a Ramsey algebra, there exists $\vec{b} \leq_{\mathcal{F}} \langle a,a,a, \dotsc \rangle$ such that $\FR_{\mathcal{F}}(\vec{b})$ is monochromatic. By the choice of our coloring, $\FR_{\mathcal{F}}(\vec{b})$ must consist of a single element, say $e$. Then $e$ is an idempotent element
for $(A,\mathcal{F})$. Furthermore, $e=\vec{b}(0)\in \FR_{\mathcal{F}}(\langle a,a,a, \dotsc \rangle)$.

($\textit{2} \Rightarrow \textit{3}$) Fix $\vec{a} \in {^\omega}\!A$. Since $A$ is finite, by going to a subsequence, we may assume that the sequence $\vec{a}$ is $\langle a,a,a,\dotsc \rangle$ for some $a \in A$. Choose an idempotent element $e$ for $(A,\mathcal{F})$ from $\FR_{\mathcal{F}}(\langle a,a,a, \dotsc \rangle)$. Take $\vec{b}$ to be the sequence $\langle e,e,e,\dotsc \rangle$. Then $\vec{b}\leq_{\mathcal{F}} \vec{a}$ and $\FR_{\mathcal{F}}(\vec{b})=\{e\}$.

($\textit{3} \Rightarrow \textit{1}$) Immediate.

($\textit{2} \Leftrightarrow \textit{4}$) This follows immediately from the observation that for any $a\in A$,  
$\FR_{\mathcal{F}}(\langle a,a,a, \dotsc \rangle)$ is the underlying set of the smallest subalgebra of $(A,\mathcal{F})$ containing (generated by) $a$.
\end{proof}

The following (see \cite[V,~\S2]{HS98}) is a consequence of Hindman's Theorem. 

\begin{theorem}\label{050613a}
Every semigroup is a Ramsey algebra. 
\end{theorem}




It is easy to show that every subalgebra of a Ramsey algebra is Ramsey. However, if $(A, \mathcal{F})$ is a  Ramsey algebra and $\mathcal{F}'\subseteq \mathcal{F}$, then $(A, \mathcal{F}')$ need not be a  Ramsey algebra. An example will be given after Theorem~\ref{1001b}.
On the other hand, if $(A, \mathcal{F})$ and $(A,\mathcal{G})$ are Ramsey algebras, then $(A, \mathcal{F}\cup \mathcal{G})$ need not be Ramsey as well. To see this, consider the algebra $(\mathbb{Z}, +, \times)$. Being an infinite integral domain, it is not Ramsey, by Theorem~\ref{0814}. However, both the semigroups $(\mathbb{Z},+)$ and $(\mathbb{Z}, \times)$ are Ramsey.

The following localized version of Ramsey algebras will be needed.

\begin{definition}
Suppose $(A,\mathcal{F})$ is an algebra and $\vec{a}$ is an infinite sequence in $A$. We say that $(A, \mathcal{F})$ is \emph{Ramsey below $\vec{a}$} if{f}
for every $\vec{b} \leq_{\mathcal{F}}\vec{a}$ and  $X \subseteq A$, there exists $\vec{c} \leq_{\mathcal{F}} \vec{b}$ homogeneous for $X$. 
\end{definition}

\begin{remark}\label{0102a}
An algebra $(A, \mathcal{F})$ is Ramsey if and only if it is Ramsey below $\vec{a}$ for every $\vec{a}\in {^\omega}\!A$. 
\end{remark}

\section{Unary Operations}

In this section, we will show that if we begin with a Ramsey algebra for which none of the operations is unary  and expand it by adding unary operations,
then the expanded algebra is Ramsey if and only if every sequence has a reduction with a certain fixed point property. 
As a corollary, a classification of Ramsey algebras for which the collection of operations are all unary is obtained.

\begin{lemma}[Katet\v{o}v \cite{mK67}]\label{0925a}
Suppose $A$ is a set and $T \colon A \rightarrow A$ is a unary operation without fixed points. Then there is a partition $A=A_1 \cupdot A_2 \cupdot A_3$ such that $A_i \cap T[A_i]=\emptyset$ for $i=1,2,3$, where $T[X]=\{\,T(x) \mid x\in X   \,\} $.
\end{lemma}


\begin{theorem}\label{0928b}
Suppose $\mathcal{F}$ is a collection of unary operations on a set $A$ and $\mathcal{G}$ is a collection of non-unary operations on $A$.
Let 
$$S=\{\,a \in A\mid f(a)=a \text{ \textnormal{for all} } f \in \mathcal{F}\,\}.$$
 Then the following are equivalent.
 \begin{enumerate}
\item $(A, \mathcal{F} \cup \mathcal{G})$ is a Ramsey algebra.
\item For every $\vec{a} \in {^\omega}\!A$, there exists $\vec{b} \leq_{\mathcal{F} \cup \mathcal{G}} \vec{a}$ with
$\FR_{ \mathcal{G}}(\vec{b})\subseteq S$ such that $(A, \mathcal{G})$ is Ramsey below $\vec{b}$.
\item For every $\vec{a} \in {^\omega}\!A$, there exists $\vec{b} \leq_{\mathcal{F} \cup \mathcal{G}} \vec{a}$ with
$\FR_{ \mathcal{G}}(\vec{b})\subseteq S$ such that for every $X\subseteq A$, there exists $\vec{c} \leq_{\mathcal{G}} \vec{b}$ homogeneous for $X$ with respect to $\mathcal{G}$.
\end{enumerate}
\end{theorem}

\begin{proof}

($\textit{2} \Rightarrow \textit{3}$) Straightforward.

($\textit{3} \Rightarrow \textit{1}$) Fix $\vec{a} \in {^\omega}\!A$ and $X \subseteq A$. Choose $\vec{b} \leq_{\mathcal{F} \cup \mathcal{G}} \vec{a}$ according to (\textit{3}). Then we can choose $\vec{c} \leq_{ \mathcal{G}} \vec{b}$  such that $\FR_{ \mathcal{G}}(\vec{c})$ is either contained in or disjoint from $X$.
By transitivity, $\vec{c} \leq_{\mathcal{F} \cup \mathcal{G}} \vec{a}$. 
Since $\FR_{ \mathcal{G}}(\vec{c})\subseteq \FR_{ \mathcal{G}}(\vec{b})\subseteq S$, it follows that $\FR_{\mathcal{F} \cup\mathcal{G}}(\vec{c})=\FR_{\mathcal{G}}(\vec{c})$. Hence, $\vec{c}$ is homogeneous for $X$ with respect to $\mathcal{F} \cup \mathcal{G}$.

($\textit{1} \Rightarrow \textit{2}$) First of all, we define a useful coloring. Let $\alpha$ be a symbol not in $A$. By the axiom of choice, choose an operation $T \colon A \cup \{\alpha\} \rightarrow A \cup \{\alpha\}$ such that 
\begin{itemize}
\item if $a \in S$, then $T(a)= \alpha$; and
\item $T(\alpha)=a_0$, where $a_0$ is some fixed element of $A$;
\item if $a \in A \cap S^c$, then $T(a)= f(a)\neq a$ for some $f \in \mathcal{F}$. 
\end{itemize}
By our choice, $T$ does not have a fixed point.
By Lemma \ref{0925a}, there exists a partition $A \cup \{\alpha\}=B_1 \cupdot B_2 \cupdot B_3$ such that $B_i \cap T[B_i]=\emptyset$ for $i=1,2,3$. This induces a partition $A=A_1 \cupdot A_2 \cupdot A_3$ such that  
whenever $a\in A\cap S^c$, if $a\in A_i$, then $f(a) \notin A_i$ for some $f \in \mathcal{F}$. 

Note that for every $\vec{b} \in {^\omega}\!A$, if $\FR_{ \mathcal{G}}(\vec{b})\subseteq S$, then 
$(A, \mathcal{G})$ is Ramsey below $\vec{b}$ 
if and only if $(A,\mathcal{F}\cup \mathcal{G})$ is Ramsey below $\vec{b}$. Hence, by Remark~\ref{0102a}, it suffices to show that for every $\vec{a} \in {^\omega}\!A$, there exists $\vec{b}  \leq_{\mathcal{F} \cup \mathcal{G}} \vec{a}$ such that $\FR_{ \mathcal{G}}(\vec{b})\subseteq S$. 
We argue by contradiction. Assume $\vec{a}\in {^\omega}\!A$ is a sequence such that $\FR_{ \mathcal{G}}(\vec{b})\nsubseteq S$ for every  $\vec{b} \leq_{\mathcal{F} \cup \mathcal{G}} \vec{a}$. 
Since  $(A, \mathcal{F} \cup \mathcal{G})$ is a Ramsey algebra, by Remark~\ref{0602a}, choose  $\vec{b} \leq_{\mathcal{F} \cup \mathcal{G}} \vec{a}$ such that $\FR_{\mathcal{F}\cup \mathcal{G}}(\vec{b})\subseteq A_i$  for some $1 \leq i \leq 3$.  By the assumption on $\vec{a}$, choose $a$ such that $a\in \FR_{\mathcal{G}}(\vec{b})\cap S^c$. Since $ \FR_{\mathcal{G}}(\vec{b})\subseteq \FR_{\mathcal{F}\cup \mathcal{G}}(\vec{b})\subseteq A_i$, thus $a \in A_i \cap S^c$. This implies that $f(a) \notin A_i$
for some $f \in \mathcal{F}$, by our choice of partition.
However, by the definition of finite reduction, since $a\in \FR_{\mathcal{G}}(\vec{b})$, it follows that $f(a)$ is in $\FR_{\mathcal{F}\cup \mathcal{G}}(\vec{b})$ and hence is in $A_i$, a contradiction.
\end{proof}


\begin{corollary}\label{1025f}
Suppose $\mathcal{F}$ is a collection of unary operations on a set $A$ and $\mathcal{G}$ is a collection of non-unary operations on $A$.
Let $S$ denote $\{\,a \in A\mid f(a)=a \text{ \textnormal{for all} } f \in \mathcal{F}\,\}$.
Assume $(A,  \mathcal{G})$ is a Ramsey algebra.   Then $(A, \mathcal{F} \cup \mathcal{G})$ is a Ramsey algebra if and only if for every $\vec{a} \in {^\omega}\!A$, there exists $\vec{b} \leq_{\mathcal{F} \cup \mathcal{G}} \vec{a}$ such that
 $\FR_{\mathcal{G}}(\vec{b})\subseteq S$. 
\end{corollary}

\begin{proof}
Straightforward by Theorem \ref{0928b}.
\end{proof}

\begin{theorem}\label{1001b}
Suppose $\mathcal{F}$ is a collection of unary operations on a set $A$. Let $S$ denote $\{\,a \in A\mid f(a)=a \text{ \textnormal{for all} } f \in \mathcal{F}\,\}$. Then the following are equivalent.
\begin{enumerate}
\item $(A, \mathcal{F})$ is a Ramsey algebra.
\item  Every element in $A$ can be sent to an element in $S$ by finitely many  applications of the unary operations in $\mathcal{F}$.
\end{enumerate}
\end{theorem}


\begin{proof}
In Example~\ref{2601a}, we have seen that the empty algebra $(A,\emptyset)$ is Ramsey; thus the conclusion of Corollary \ref{1025f} holds with $\mathcal{G}=\emptyset$. 

(\textit{1} $\Rightarrow$ \textit{2}) Suppose $a \in A$. By Corollary~\ref{1025f}, 
choose $\vec{b} \leq_{\mathcal{F}} \langle a, a,a, \dotsc \rangle$   such that $\FR_{\emptyset}(\vec{b})\subseteq S$. 
From the definition of reduction, $\vec{b}(0)$ is obtained from $a$ by finitely many  applications of the unary operations in $\mathcal{F}$.
Additionally, $\vec{b}(0)\in \FR_{\emptyset}(\vec{b})\subseteq S$. 
 
(\textit{2} $\Rightarrow$ \textit{1}) The assumption implies that for every $\vec{a} \in {^\omega}\!A$, there exists $\vec{b} \in {^\omega}\!S$ such that $\vec{b} \leq_{\mathcal{F}} \vec{a}$. However, $\vec{b} \in {^\omega}\!S$ implies that $\FR_{\emptyset}(\vec{b})\subseteq S$. Therefore, by Corollary~\ref{1025f},
$(A, \mathcal{F})$ is a Ramsey algebra.
\end{proof}

\begin{example}\label{0809e}
Let $f$ and $g$ be the following unary operations on $\omega$.
$$f(x)= \begin{cases}
x,  &\text{if } x\in 3\omega\\
x-1,  &\text{if } x\in 3\omega+1\\
x+3, &\text{otherwise }
\end{cases} \qquad 
g(x)= \begin{cases}
x,  &\text{if } x\in 3\omega\\
x+3,  &\text{if } x\in 3\omega+1\\
x-2, &\text{otherwise }
\end{cases}$$
Using Theorem~\ref{1001b}, it is easy to verify  that neither $(\omega,f)$ nor $(\omega,g)$ is a Ramsey algebra but $(\omega,f,g)$ is.
\end{example}

\section{Rings}\label{190813}

For us a ring is not assumed to have a multiplicative identity. By Theorem~\ref{0808}, every finite ring is a Ramsey algebra because the zero element is contained in any subalgebra. On the contrary, we will show that no infinite ring without  zero divisors is a Ramsey algebra. 
First of all, monotone polynomials will be defined. For any ring, the orderly terms over the ring operations will be shown to be 
defined by these monotone polynomials.

Suppose $x_1, x_2, x_3, \dotsc$ is a fixed list of variables.
For every nonempty finite strictly increasing sequence $\alpha$ of positive integers, let $x^\alpha$ denote the expression $x_{\alpha(0)}x_{\alpha(1)}\dotsm \,x_{\alpha(\vert \alpha\vert -1)}$. (For notational brevity, we suppress the arrow above $\alpha$.) We say that $x^{\alpha}$ is a \emph{monotone monomial}.
A \emph{monotone  polynomial}  is a formal sum of the form $\sum_{\alpha \in M} x^\alpha$, where $M$ is a finite set of nonempty finite strictly increasing sequences of positive integers. In other words, a monotone polynomial is a  finite nonrepetitive sum of monotone monomials. If $M$ is empty, we say that $\sum_{\alpha \in M} x^\alpha$ is the \emph{zero polynomial}. Notice that  $\sum_{\alpha \in M} x^\alpha$ can be identified with $M$.
A variable $x_i$ is said to appear in $\sum_{\alpha \in M} x^\alpha$  if and only if $i$ appears in $\alpha$ for some $\alpha \in M$.
Let $\MP_n$ denote the set of monotone  polynomials $\sum_{\alpha \in M} x^\alpha$ such that the variables appearing in $\sum_{\alpha \in M} x^\alpha$ are among  $x_1, \dotsc, x_n$. Obviously,  $\MP_1 \subseteq \MP_2 \subseteq \MP_3 \subseteq \dotsb$. 


Suppose $\mathfrak{A}= (A,+,\cdot) $ is a ring and suppose $\sum_{\alpha \in M} x^\alpha\in \MP_n$. Then $\sum_{\alpha \in M} x^\alpha$ defines naturally an $n$-ary operation 
$\sum_{\alpha \in M} x^\alpha \colon A^n \rightarrow A$ given by
$$(\sum\nolimits_{\alpha \in M} x^\alpha  )(a_1, \dotsc, a_n)= \sum_{\alpha \in M }a_{\alpha(0)}a_{\alpha(1)}\dotsm \,a_{\alpha(\vert \alpha\vert -1)}.$$
Since $\mathfrak{A}$ is a ring, the product $a_{\alpha(0)}a_{\alpha(1)}\dotsm \,a_{\alpha(\vert \alpha\vert -1)}$ is unambiguous and the order of the summation is irrelevant. By convention, the zero polynomial evaluates everywhere to the zero element $0_{\mathfrak{A}}$ of the ring. As a matter of taste,
let $\sum_{\alpha \in M } (a_1, \dotsc, a_n)^\alpha$ denote $(\sum\nolimits_{\alpha \in M} x^\alpha  )(a_1, \dotsc, a_n)$.

\begin{lemma}\label{1025a}
Suppose $(A,+,\cdot)$ is a ring. For every orderly term $f$ over $\{+,\cdot\}$, if $f$ is $n$-ary, then there exists a monotone polynomial $\sum_{\alpha \in M} x^\alpha$ such that
\begin{enumerate}
\item the variables appearing in $\sum_{\alpha \in M} x^\alpha$ are exactly $x_1, x_2, \dotsc, x_n$;
\item $f(a_1, a_2,\dotsc,a_n)   =  \sum_{\alpha \in M} (a_1, a_2, \dotsc, a_n)^\alpha$ for all  $a_1, a_2, \dotsc,a_n \in A$.
\end{enumerate}
\end{lemma}

\begin{proof}
The proof is by induction on the generation of the orderly terms over $\{+,\cdot\}$.
Clearly, the identity function on $A$ corresponds to the monotone monomial $x_1$.
Similarly, the ring addition corresponds  to the monotone polynomial $x_1+x_2$
and the ring multiplication corresponds to the monotone monomial $x_1x_2$.

Suppose $f$ is given by $f(\vec{x}_1\ast\vec{x}_2)=g(h_1(\vec{x}_1), h_2(\vec{x}_2))$, where $g \in \{+,\cdot\}$ while $h_1$ and $h_2$ are $m$-ary and $n$-ary orderly terms over $\{+,\cdot\}$ respectively.
By the induction hypothesis, for some monotone polynomial $\sum_{\alpha \in M_1} x^\alpha$ such that the variables appearing in it are exactly $x_1, x_2, \dotsc, x_m$,
$$h_1(a_1, a_2,\dotsc,a_m)   =  \sum_{\alpha \in M_1} (a_1, a_2, \dotsc, a_m)^\alpha$$
for all $a_1, a_2, \dotsc,a_m \in A$.
Similarly, after reindexing the variables,  for some monotone polynomial $\sum_{\alpha \in M_2} x^\alpha$ such that the variables appearing in it are exactly $x_{m+1}, x_{m+2}, \dotsc, x_{m+n}$, 
$$h_2(a_{m+1}, a_{m+2}, \dotsc, a_{m+n})   =  \sum_{\alpha \in M_2} (a_1, a_2, \dotsc, a_{m+n})^\alpha$$
for all $a_1, a_2,  \dotsc, a_{m+n} \in A$.

\emph{Case} 1. If $g$ is the ring addition, then 
\begin{align*}
f(a_1, a_2,\dotsc, a_{m+n})=& \sum_{\alpha \in M_1} (a_1, a_2, \dotsc, a_m)^\alpha+ \sum_{\alpha \in M_2} (a_1, a_2, \dotsc, a_{m+n})^\alpha\\
=& \sum_{\alpha \in M_1\cup M_2} (a_1, a_2,\dotsc, a_{m+n})^\alpha \qquad\; \text{($\because$  $M_1$ and $M_2$ are disjoint)}
\end{align*}
for all $a_1, a_2,\dotsc, a_{m+n} \in A$. 

\emph{Case} 2. If $g$ is the ring multiplication, then 
\begin{align*}
f(a_1, a_2,\dotsc, a_{m+n})  =& \sum_{\alpha \in M_1} (a_1, a_2, \dotsc, a_m)^\alpha\cdot \sum_{\alpha \in M_2} (a_1, a_2, \dotsc, a_{m+n})^\alpha\\
=&\sum_{\alpha \in M_1\ast M_2} (a_1, a_1,\dotsc, a_{m+n})^\alpha
\end{align*}
for all $a_1, a_1,\dotsc, a_{m+n} \in A$, where
$M_1 \ast M_2=\{\alpha_1\ast \alpha_2 \mid \alpha_1 \in M_1 \text{ and } \alpha_2 \in M_2  \}$.

In either case, the variables appearing in the corresponding monotone polynomial are exactly $x_1, x_2,\dotsc, x_{m+n} $.
\end{proof}

\begin{remark}
In fact, the monotone polynomial in Lemma~\ref{1025a} can be chosen uniformly. It depends on the path of generation of a given orderly term and not on the specific ring operations.
In other words, it depends on the corresponding formal orderly term (see \cite{wcT13a}) in the language of rings. The proof then goes by induction on the complexity of formal orderly terms.
\end{remark}

\begin{lemma}\label{1025b}
Suppose  $\mathfrak{A}=(A,+,\cdot)$  is an infinite ring without zero divisors. There exists a sequence $\langle a_i \rangle_{i \geq 1} \in {^\omega}\!A$ such that for every
$n \geq 1$, whenever $\sum_{\alpha \in M} x^\alpha$, $\sum_{\alpha \in N} x^\alpha \in \MP_n$ are distinct, $\sum_{\alpha \in M} (a_1,a_2, \dotsc, a_n)^\alpha \neq \sum_{\alpha \in N} (a_1,a_2, \dotsc, a_n)^\alpha$.
\end{lemma}

\begin{proof}
In fact, we will construct a sequence $\langle a_i \rangle_{i \geq 1} \in {^\omega}\!A$ inductively such that for every $n \geq 1$, whenever $\sum_{\alpha \in M} x^\alpha, \sum_{\alpha \in N} x^\alpha \in \MP_n$ are distinct, $\sum_{\alpha \in M} (a_1, \dotsc, a_n)^\alpha \neq \sum_{\alpha \in N} (a_1, \dotsc, a_n)^\alpha$ and $\sum_{\alpha \in M} (a_1, \dotsc, a_n)^\alpha - \sum_{\alpha \in N} (a_1, \dotsc, a_n)^\alpha$ is not a  multiplicative identity. 

Clearly, if $\sum_{\alpha \in M} x^\alpha \in \MP_1$ then $\sum_{\alpha \in M} x^\alpha$ is either the zero polynomial or $x_1$. Simply choose $a_1$ to be any nonzero element of the ring that is neither  $1_{\mathfrak{A}}$ nor  $-1_{\mathfrak{A}}$   (if the ring has a multiplicative identity). 
Assume $a_1,a_2, \dotsc, a_n$ have been constructed satisfying the induction hypothesis.

\begin{claim} Suppose $\sum_{\alpha \in M} x^\alpha, \sum_{\alpha \in N} x^\alpha \in \MP_{n+1}$ are distinct.
For $\epsilon$ equals to either $0_{\mathfrak{A}}$ or $1_{\mathfrak{A}}$ (if the ring has a multiplicative identity),
there is at most  one $b \in A$ such that 
$\sum_{\alpha \in M} (a_1, \dotsc, a_n,b)^\alpha = \sum_{\alpha \in N} (a_1, \dotsc, a_n,b)^\alpha +\epsilon$. 
\end{claim}

Assuming the claim, since the set $\MP_{n+1}$ is finite and the ring is infinite, we can choose $a_{n+1}$ so that it avoids the finitely many $b$ that would violates the requirement on $a_{n+1}$. This completes the induction step.

Now, we prove the claim, using only elementary algebraic arguments. 
Given distinct $\sum_{\alpha \in M} x^\alpha, \sum_{\alpha \in N} x^\alpha \in \MP_{n+1}$ and $\epsilon$, which is equal to either $0_{\mathfrak{A}}$ or $1_{\mathfrak{A}}$.
Let $M'$ denote $\{\, \alpha \in M \mid   \alpha(\vert \alpha\vert-1)<n+1\,\}$ and $M''$ denote $\{\, \alpha \in {^{<\omega}}\mathbb{N} \mid   \alpha \ast \langle n+1\rangle \in M \text{ and } \alpha \neq \langle \phantom{1}\rangle\,\}$. 
Similarly, let $N'$ denote $\{\, \alpha \in N \mid   \alpha(\vert \alpha\vert-1)<n+1\,\}$ and $N''$ denote $\{\, \alpha \in {^{<\omega}}\mathbb{N} \mid   \alpha \ast \langle n+1\rangle \in N \text{ and } \alpha \neq \langle \phantom{1}\rangle\,\}$.
Let $\sigma_{M'}, \sigma_{M''},\sigma_{N'}, \sigma_{N''}$ stand for the ring elements $\sum_{\alpha \in M'}  (a_1, \dotsc, a_n)^\alpha$, $\sum_{\alpha \in M''}  (a_1, \dotsc, a_n)^\alpha$, $\sum_{\alpha \in N'}  (a_1, \dotsc, a_n)^\alpha$, $\sum_{\alpha \in N''}  (a_1, \dotsc, a_n)^\alpha $ respectively.
Using the ring properties, for every $b \in A$,
\begin{align*}
\sum_{\alpha \in M} (a_1, \dotsc, a_n,b)^\alpha&=
\begin{cases}
\displaystyle \sigma_{M'} +\sigma_{M''}b &\text{ if } \langle n+1\rangle \notin M\\
\displaystyle \sigma_{M'} +\sigma_{M''}b +b &\text{ if } \langle n+1\rangle \in M
\end{cases},\; \text{and}\\
\sum_{\alpha \in N} (a_1, \dotsc, a_n,b)^\alpha&=
\begin{cases}
\displaystyle \sigma_{N'} +\sigma_{N''}b &\text{ if } \langle n+1\rangle \notin N \\
\displaystyle \sigma_{N'} +\sigma_{N''}b +b &\text{ if } \langle n+1\rangle \in N
\end{cases}.
\end{align*}

We argue by contradiction. Assume for some distinct $b_1, b_2 \in A$,
\begin{align*}
\sum_{\alpha \in M} (a_1, \dotsc, a_n,b_1)^\alpha=&\sum_{\alpha \in N} (a_1, \dotsc, a_n,b_1)^\alpha+\epsilon, \; \text{and}    \\
\sum_{\alpha \in M} (a_1, \dotsc, a_n,b_2)^\alpha=&\sum_{\alpha \in N} (a_1, \dotsc, a_n,b_2)^\alpha+\epsilon.
\end{align*}

\emph{Case} 1. $\langle n+1 \rangle \in M \leftrightarrow \langle n+1\rangle \in N$.\\
From the two equations above, it follows that $(\sigma_{M''}-\sigma_{N''})(b_1-b_2)=0_{\mathfrak{A}}$. Since the ring has no zero divisors and $b_1-b_2\neq 0_{\mathfrak{A}}$,
this implies $\sigma_{M''}=\sigma_{N''}$. This in turn implies that
$\sigma_{M'}=\sigma_{N'}+\epsilon$ from either of the above equations. However, since  $\sum_{\alpha \in M'} x^\alpha$, $\sum_{\alpha \in N'} x^\alpha  \in \MP_{n}$, by the induction hypothesis, they cannot be distinct and thus $M'=N'$. Similarly, $M''=N''$ and thus $M=N$. Therefore, $\sum_{\alpha \in M} x^\alpha$ and $\sum_{\alpha \in N} x^\alpha$ are indistinct, a contradiction.
  
\emph{Case} 2. We may assume $\langle n+1 \rangle \in N\backslash M$ as the case  $\langle n+1 \rangle \in M\backslash N$  is similar. 
This time, it follows that $(\sigma_{M''}-\sigma_{N''})(b_1-b_2)=b_1-b_2$.
Since the ring has no zero divisors and $b_1-b_2\neq 0_{\mathfrak{A}}$, it can be deduced that $\sigma_{M''}-\sigma_{N''}$ is a multiplicative identity. (First, show that $\sigma_{M''}-\sigma_{N''}$ is a right multiplicative identity.) However, this contradicts the induction hypothesis because $\sum_{\alpha \in M''} x^\alpha$, $\sum_{\alpha \in N''} x^\alpha  \in \MP_{n}$ and they are distinct. 
\end{proof}

\begin{lemma}\label{0816}
Suppose  $(A,+,\cdot)$ is an infinite ring without zero divisors. Then there exists a sequence $\vec{a}\in {^\omega}\! A$ with the property that 
$f(\vec{a}_1)+g(\vec{a}_2)\neq f'(\vec{a}_1')\cdot g'(\vec{a}_2')$ whenever $f,g,f',g'$ are orderly terms over $\{+,\cdot\}$ and $\vec{a}_1\ast\vec{a}_2$  and $\vec{a}_1'\ast\vec{a}_2'$ are finite subsequences of $\vec{a}$.
\end{lemma}


\begin{proof}
Take $\vec{a}$ to be any sequence $\langle a_i \rangle_{i\geq 1}$ having the property as stated in Lemma~\ref{1025b}. Suppose $f,g,f',g'$ 
are orderly terms over $\{+,\cdot\}$. Suppose $\vec{a}_1,\vec{a}_2,\vec{a}_1',\vec{a}_2'$  are finite subsequences of $\vec{a}$, whose length are the arity of $f,g,f',g'$ respectively, such that  $\vec{a}_1\ast\vec{a}_2$  and $\vec{a}_1'\ast\vec{a}_2'$ are subsequences of an initial segment of  $\vec{a}$, say of length $N$.
Say that $\vec{a}_1=\langle a_{i_1}, \dotsc, a_{i_m}\rangle$ and
$\vec{a}_2=\langle a_{j_1}, \dotsc, a_{j_n}\rangle$, where $i_1<\dotsb< i_m <j_1< \dotsb< j_n$.
 By Lemma \ref{1025a}, after reindexing the variables, we have $f(\vec{a}_1)  =  \sum_{\alpha \in M_f} (a_1, a_2, \dotsc, a_N)^\alpha$ for some monotone polynomial $\sum_{\alpha \in M_f} x^\alpha$ such that the variables appearing in it are exactly $x_{i_1}, x_{i_2}, \dotsc, x_{i_m}$
 and $g(\vec{a}_2)  =  \sum_{\alpha \in M_g} (a_1, a_2, \dotsc, a_N)^\alpha$ for some monotone polynomial $\sum_{\alpha \in M_g} x^\alpha$ such that the variables appearing in it are exactly $x_{j_1}, x_{j_1}, \dotsc, x_{j_n}$. Similarly, we obtain $M_{f'}$ and $M_{g'}$. Now, 
$f(\vec{a}_1)+g(\vec{a}_2)=\sum_{\alpha \in M_f\cupdot M_g} (a_1, a_2, \dotsc, a_N)^\alpha$ and
$f'(\vec{a}_1')\cdot g'(\vec{a}_2')=\sum_{\alpha \in M_{f'}\ast M{g'}} (a_1, a_2, \dotsc, a_N)^\alpha$,\linebreak
where  $M_{f'}\ast M_{g'}=\{\alpha_1\ast \alpha_2 \mid \alpha_1 \in M_{f'} \text{ and } \alpha_2 \in M_{g'}  \}$.
It is not hard to see that $\sum_{\alpha \in M_f\cupdot M_g} x^{\alpha}$ and $\sum_{\alpha \in M_{f'}\ast M{g'}}x^{\alpha}$ are distinct monotone polynomials.
Hence, by our choice of $\vec{a}$, $\sum_{\alpha \in M_f\cup M_g} (a_1, \dotsc, a_N)^\alpha \neq \sum_{\alpha \in M_{f'}\ast M{g'}} (a_1, \dotsc, a_N)^\alpha$.
\end{proof}

\begin{theorem}\label{0814}
No infinite ring without zero divisors is a Ramsey algebra. In particular, no infinite integral domain is a Ramsey algebra.
\end{theorem}

\begin{proof}
Suppose $(A, \mathcal{F})$  is an  infinite  ring without zero divisors, where $\mathcal{F}=\{+,\cdot\}$.
Choose a sequence $\vec{a}$ having the property stated in Lemma~\ref{0816}. Let $X$ be the set of all elements in $A$ of the form $f(\vec{a}_1)+g(\vec{a}_2)$, where $f$ and $g$ are orderly terms over $\mathcal{F}$ and $\vec{a}_1\ast\vec{a}_2$ is a finite subsequence of $\vec{a}$. Suppose $\vec{b} \leq_{\mathcal{F}} \vec{a}$. Then by the definition of reduction, 
$\vec{b}(0)=f(\vec{a}_1   )$ and $\vec{b}(1)= g( \vec{a}_2 )$ for some orderly terms $f,g$ over $\mathcal{F}$ and finite sequences $\vec{a}_1,\vec{a}_2$ such that
$\vec{a}_1\ast \vec{a}_2$ is a finite subsequence of $\vec{a}$. Hence, $\vec{b}(0)+\vec{b}(1) \in X$. On the other hand, $\vec{b}(0)\cdot \vec{b}(1) =  f(\vec{a}_1  )\cdot g( \vec{a}_2 )\notin X$ by the choice of $\vec{a}$. 
However, both $\vec{b}(0)+\vec{b}(1) $ and $\vec{b}(0)\cdot \vec{b}(1) $ are finite reductions of $\vec{b}$.
It follows that neither $\FR_{\mathcal{F}}(\vec{b}) \subseteq X$ nor $\FR_{\mathcal{F}}(\vec{b}) \subseteq X^c$. Hence, no reduction of $\vec{a}$ is homogeneous for $X$.
Therefore, $(A, \mathcal{F})$ is not a Ramsey algebra.
\end{proof}

\begin{corollary}
No infinite ring with multiplicative identity having characteristic zero is a Ramsey algebra.
\end{corollary}

\begin{proof}
If a ring has a multiplicative identity and has characteristic zero, then we can identify $(\mathbb{Z},+,\times)$ as a subring, and hence as a subalgebra. By Theorem~\ref{0814}, $(\mathbb{Z},+,\times)$ is not a Ramsey algebra. 
Since every subalgebra of a Ramsey algebra is Ramsey, the result follows.
\end{proof}


\begin{example}
The assumption that the ring has no zero divisors in Theorem~\ref{0814} is necessary. 
Let  $\mathfrak{A}=( \oplus_{n=1}^{\infty} \mathbb{F}_2 , +_{\mathfrak{A}},\cdot_{\mathfrak{A}}   )$  be the infinite direct sum  of the field of two elements $\mathbb{F}_2=\{0,1\}$. Clearly, $\mathfrak{A}$ is  an  infinite ring with zero divisors, and its zero element $0_{\mathfrak{A}}$ is $(0,0,\dotsc)$. 
We will show that $\langle 0_{\mathfrak{A}}, 0_{\mathfrak{A}},  \dotsc \rangle$ is a reduction of every $\vec{a}\in {^\omega} (\oplus_{n=1}^{\infty} \mathbb{F}_2)$. Hence, $\mathfrak{A}$ is in  fact a degenerate Ramsey algebra.
Suppose  $\vec{a} \in {^\omega} (\oplus_{n=1}^{\infty} \mathbb{F}_2)$. By the definition of direct sum, the nonzero entries of $\vec{a}(0)$ all appear within the first $N$  entries for some $N$. By the pigeonhole principle, there are $0<i_0<j_0$ such that $\vec{a}(i_0)$ and $\vec{a}(j_0)$  agree on the first $N$ entries. It is then clear that  $\vec{a}(0)\cdot_{\mathfrak{A}} (\vec{a}(i_0) +_{\mathfrak{A}} \vec{a}(j_0))=0_{\mathfrak{A}}$. 
Arguing similarly with $\vec{a}(j_0+1)$ taking the role of  $\vec{a}(0)$, there are $j_0+1 <i_1<j_1$ such that $\vec{a}(j_0+1) \cdot_{\mathfrak{A}}(\vec{a}(i_1) +_{\mathfrak{A}} \vec{a}(j_1))=0_{\mathfrak{A}}$. Proceed inductively. Since the operation $f(x,y,z)=x\cdot_{\mathfrak{A}} (y+_{\mathfrak{A}}z)$ is an orderly term over $\{+_{\mathfrak{A}},\cdot_{\mathfrak{A}}\}$, therefore, $\langle 0_{\mathfrak{A}}, 0_{\mathfrak{A}}, \dotsc \rangle$ is a reduction of $\vec{a}$.
\end{example}

\section{Ramsey Spaces}\label{220813a}
 
Here the definition of Ramsey spaces will be provided in a form that serves our purpose.
This equivalent version was already pointed out in \cite{tC88}.
Our terminology here follows closely that provided there.

\begin{definition}
A \emph{preorder with approximations}\footnote{In \cite{tC88} a preorder was called  a pre-partial order. Hence, a preorder with approximations was originally called a pre-partial order with appromixations by Carlson.} is a pair
$\mathfrak{R}=(R, \leq)$ such that $R$ is a nonempty set of infinite sequences and $\leq$ is a preorder on $R$. 
If $\vec{a},\vec{b} \in R$ and $\vec{b} \leq \vec{a}$, then $\vec{b}$ is called a \emph{reduction} of $\vec{a}$ with respect to $\mathfrak{R}$. 
\end{definition}


\begin{definition}
Suppose $\mathfrak{R}=(R,\leq)$ is a preorder with approximations. For $n \in \omega$ and $\vec{a}\in R$, define 
$$[n,\vec{a}]_{\mathfrak{R}}=\{\,\vec{b}\in R\mid \vec{b}\leq \vec{a} \text{ and } \vec{b}\!\upharpoonright \! n=\vec{a}\!\upharpoonright \! n\,\}.$$
The \emph{natural topology} on $\mathfrak{R}$ is the topology generated by the sets $[n,\vec{a}]_{\mathfrak{R}}$.
\end{definition}

\begin{example}
Suppose $R$ is the set of strictly increasing infinite sequences of natural numbers and define $\vec{a}\leq \vec{b}$ if{f} $\vec{a}$ is a subsequence of $\vec{b}$ for every $\vec{a}, \vec{b}\in R$. Via identification of infinite subsets of natural numbers with elements of $R$, the preorder with approximations $(R,\leq)$, endowed with the natural topology, is the so-called Ellentuck's space.
\end{example}

\begin{remark}\label{1706a}
Suppose $(R, \leq)$ is a preorder with approximations and $T$ is a subset of $R$. Abusing notation, obviously $\mathfrak{T}=(T, \leq )$ is a subpreorder with approximations, where the latter $\leq$ is the restriction of the former preorder $\leq$ on $R$ to $T$. The natural topology on $\mathfrak{T}$ is in fact the same as the subspace topology on $T$ induced from the natural topology on $\mathfrak{R}$. 
\end{remark}

\begin{definition}
Suppose $\mathfrak{R}= (R, \leq)$ is a preorder with approximations. Assume $X$ is a subset of $R$. We say that $X$ is \emph{Ramsey} if{f} for every $n \in \omega$ and $\vec{a} \in R$, there exists $\vec{b} \in [n,\vec{a}]_{\mathfrak{R}}$ such that $[n,\vec{b}]_{\mathfrak{R}}$ is either contained in or disjoint from $X$. Assuming the Axiom of Choice, $\mathfrak{R}$ is a \emph{Ramsey space} if{f} every subset of $\mathfrak{R}$ (endowed with the natural topology) which has the property of Baire is Ramsey.
\end{definition} 

Here are preorders with approximations induced by algebras that concern us.

\begin{definition}
Suppose $(A,\mathcal{F})$ is an algebra. Define $\mathscr{R}(A,\mathcal{F})$ to be the preorder with approximations $({^\omega}\!A, \leq_{\mathcal{F}})$ and $\mathscr{R}_{\vec{a}}(A,\mathcal{F})$ 
to be  the subpreorder with approximations $(\{\, \vec{b}\in {^\omega}\!A\mid \vec{b} \leq_{\mathcal{F}}\vec{a}\,\},\leq_{\mathcal{F}})$ for each $\vec{a}\in {^\omega}\!A$.
\end{definition}

Suppose $\mathfrak{R} =(R, \leq)  $ is a preorder with approximations. Generally, a subspace of $\mathfrak{R}$ induced by a subset $T$ of $R$ need not be Ramsey when $\mathfrak{R}$ is a Ramsey space. However, $\mathfrak{R}$ is a Ramsey space if and only if every subspace of 
$\mathfrak{R}$ induced by a basic open set $[n,\vec{a}]_{\mathfrak{R}}$ is Ramsey. 
Particularly, if $\mathscr{R}(A, \mathcal{F})$ is a Ramsey space, then $\mathscr{R}_{\vec{a}}(A, \mathcal{F})$ is a Ramsey space for every $\vec{a} \in {^\omega}\!A$. The converse also holds and can be strengthened to the following proposition, noting that assuming $\mathscr{R}_{\vec{b}}(A,\mathcal{F})$ is a Ramsey space and  $\vec{b} \leq_{\mathcal{F}} \vec{a}$, it  does not necesarily follow that $\mathscr{R}_{\vec{a}}(A,\mathcal{F})$ is a Ramsey space.

\begin{proposition}\label{1404a}
Suppose $(A,\mathcal{F})$ is an algebra. Assume that
for every $\vec{a} \in {^\omega}\!A$, there exists $\vec{b} \leq_{\mathcal{F}} \vec{a}$ such that $\mathscr{R}_{\vec{b}}(A,\mathcal{F})$ is a Ramsey space.
Then $\mathscr{R}(A,\mathcal{F})$ is a Ramsey space.
\end{proposition}

\begin{proof}
Suppose $X \subseteq {^\omega}\!A$ has the property of Baire in the natural topology on $\mathscr{R}(A,\mathcal{F})$.
Fix $n \in \omega$ and $\vec{a} \in {^\omega}\!A$.
By the hypothesis, choose $\vec{b}  \leq_{\mathcal{F}} \vec{a}-n$ such that $\mathscr{R}_{\vec{b}}(A,\mathcal{F})$ is a Ramsey space.
Let $\mathfrak{T}$ denote the subpreorder with approximations $(  \{\,\vec{a}\!\upharpoonright \! n \ast \vec{d}\mid \vec{d}\leq_{\mathcal{F}}\vec{b}\,\}, \leq_{\mathcal{F}})$ of $\mathscr{R}(A,\mathcal{F})$. It can be verified (but tedious) that the map $H \colon \{\, \vec{d}\in {^\omega}\!A\mid \vec{d} \leq_{\mathcal{F}}\vec{b}\,\} \rightarrow \{\,\vec{a}\!\upharpoonright \! n \ast \vec{d}\mid \vec{d}\leq_{\mathcal{F}}\vec{b}\,\}$ defined by
$H(\vec{d})=\vec{a}\!\upharpoonright \! n\ast\vec{d}$ is a homeomorphism between $\mathscr{R}_{\vec{b}}(A,\mathcal{F})$ and $\mathfrak{T}$ under the natural topology.

Let $T$ denote $\{\,\vec{a}\!\upharpoonright \! n \ast \vec{d}\mid \vec{d}\leq_{\mathcal{F}}\vec{b}\,\}$.  By elementary topology, $X \cap T$ has the property of Baire in the  subspace topology on $T$ induced from the natural topology on $\mathfrak{R}(A,\mathcal{F})$. Hence, by Remark~\ref{1706a}, $X\cap T$ has the property of Baire in the natural topology on $\mathfrak{T}$.
By the homeomorphism, $H^{-1}[X\cap T]$ has the property of Baire in the natural topology on $\mathscr{R}_{\vec{b}}(A,\mathcal{F})$.
Since $\mathscr{R}_{\vec{b}}(A,\mathcal{F})$ is a Ramsey space, there exists $\vec{c}\in [0, \vec{b}]_{\mathscr{R}_{\vec{b}}(A,\mathcal{F})} $ such that 
$[0, \vec{c}]_{\mathscr{R}_{\vec{b}}(A,\mathcal{F})}$ is either contained in or disjoint from $H^{-1}[X\cap T]$. 

Clearly, $\vec{a}\!\upharpoonright \! n\ast \vec{c} \leq_{\mathcal{F}} \vec{a}$ because
$\vec{c}\leq_{\mathcal{F}}\vec{b}  \leq_{\mathcal{F}} \vec{a}-n$.
Hence, $\vec{a}\!\upharpoonright \! n\ast \vec{c} \in [n,\vec{a}]_{\mathscr{R}(A,\mathcal{F})}$.
It remains to show that $[n,\vec{a}\!\upharpoonright \! n\ast \vec{c}]_{\mathscr{R}(A,\mathcal{F})}$ is either contained in or disjoint for $X$.
Now, $H([0, \vec{c}]_{\mathscr{R}_{\vec{b}}(A,\mathcal{F})}    )$ is either contained in or disjoint from $X\cap T$ and thus from $X$ because $H([0, \vec{c}]_{\mathscr{R}_{\vec{b}}(A,\mathcal{F})}    )\subseteq T$.
The proof is complete since $H([0, \vec{c}]_{\mathscr{R}_{\vec{b}}(A,\mathcal{F})}    )
= H( \{\, \vec{d}\in {^\omega}\!A \mid \vec{d}\leq_{\mathcal{F}}\vec{c}    \}      )
= \{\,\vec{a}\!\upharpoonright \! n \ast \vec{d}\mid \vec{d}\leq_{\mathcal{F}}\vec{c}\,\}
=[n, \vec{a}\!\upharpoonright \! n\ast \vec{c}]_{\mathscr{R}(A,\mathcal{F})}
$. (The equality of the corresponding sets in the last sentence can be shown using some elementary properties of $\leq_{\mathcal{F}}$, for examples, transitivity of $\leq_{\mathcal{F}}$ and if $\vec{d}\leq_{\mathcal{F}}\vec{c}$, then
$\vec{d}-n\leq_{\mathcal{F}}\vec{c}-n$ for every $n\in \omega$.)
\end{proof}

\begin{remark}\label{130813a}
If $\mathscr{R}(A, \mathcal{F})$ is a Ramsey space, then $(A, \mathcal{F})$  is immediately a Ramsey algebra.
This follows mainly  because every subset of $A$ induces a clopen set in $\mathscr{R}(A, \mathcal{F})$.
\end{remark}

As mentioned in the introduction, the abstract Ellentuck's Theorem by Carlson sorts out Ramsey spaces in terms of some combinatorial properties. In the case of $\mathscr{R}(A, \mathcal{F})$, it boils down to the following theorem.

\begin{theorem}[Carlson]\label{160713}
Suppose $\mathcal{F}$ is a finite collection of operations on a set $A$, none of which is unary. Then $\mathscr{R}(A, \mathcal{F})$ is a Ramsey space if and only if $(A, \mathcal{F})$  is a Ramsey algebra. 
\end{theorem}

\begin{proof}
This theorem is a special case of Lemma 4.14 in \cite{tC88}.
\end{proof}

Analogously, the following can be shown, which we state without proof.

\begin{theorem}\label{1129b}
Suppose $\mathcal{F}$ is a finite collection of operations on a set $A$, none of which is unary, and  $\vec{a}\in {^\omega}\!A$. Then $\mathscr{R}_{\vec{a}}(A, \mathcal{F})$ is a Ramsey space if and only if $(A, \mathcal{F})$  is Ramsey below $\vec{a}$. 
\end{theorem}

The following theorem says that the collection of operations in Theorem~\ref{160713} can actually be expanded to include arbitrary collection of unary operations.

\begin{theorem}\label{160713a}
Suppose $\mathcal{F}$ is a finite collection of non-unary operations on a set $A$ and $\mathcal{G}$ is any collection of unary operations on $A$.
Then $\mathscr{R}(A, \mathcal{F} \cup \mathcal{G})$ is a Ramsey space if and only if $(A, \mathcal{F}\cup \mathcal{G})$ is  a Ramsey algebra.
\end{theorem}

\begin{proof}
($\Rightarrow$) This is immediate by Remark \ref{130813a}.

($\Leftarrow$) Fix $\vec{a} \in {^\omega}\!A$. By Theorem \ref{0928b} (beware of the opposite roles of $\mathcal{F}$ and $\mathcal{G}$), choose $\vec{b} \leq_{\mathcal{F} \cup \mathcal{G}} \vec{a}$ with
 $\FR_{ \mathcal{F}}(\vec{b})\subseteq S$, where $S= \{\,a \in A\mid g(a)=a \text{ \textnormal{for all} } g \in \mathcal{G}\,\}$ such that $(A, \mathcal{F})$ is Ramsey below $\vec{b}$. By Theorem \ref{1129b}, 
 $\mathscr{R}_{\vec{b}}(A,  \mathcal{F})$ is a Ramsey space. 
 Since $\FR_{ \mathcal{F}}(\vec{b})\subseteq S$, the set $\{\, \vec{c}\in {^\omega}\!A\mid \vec{c} \leq_{\mathcal{F} \cup \mathcal{G}}\vec{b}\,\}$ is equal to
  $\{\, \vec{c}\in {^\omega}\!A\mid \vec{c} \leq_{\mathcal{F}}\vec{b}\,\}$, and the restrictions of $\leq_{\mathcal{F} \cup \mathcal{G}     }$ and  $\leq_{\mathcal{F}} $ to it coincide. Therefore, $\mathscr{R}_{\vec{b}}(A, \mathcal{F}\cup  \mathcal{G})$ and $\mathscr{R}_{\vec{b}}(A,  \mathcal{F})$ are in fact the same preorder with approximations. Thus $\mathscr{R}_{\vec{b}}(A,  \mathcal{F}\cup \mathcal{G})$ is a Ramsey space. Since $\vec{a}$ is arbitrary, by Proposition~\ref{1404a}, we conclude that $\mathscr{R}(A, \mathcal{F} \cup \mathcal{G})$ is a Ramsey space.
\end{proof}


An analogous notion of reduction  can be defined between finite sequences. The restriction on $\mathcal{F}$ in Theorem~\ref{160713} ensures that
every finite sequence has at most finitely many reductions, thus 
fulfilling one of the criteria to apply the abstract Ellentuck's theorem on $\mathscr{R} (A,\mathcal{F})$.  
We do not know of an example of a Ramsey algebra $(A,\mathcal{F})$ such that $\mathscr{R}(A, \mathcal{F})$ is not a Ramsey space. One may start out with a semigroup and try to expand the algebra by adding some unary functions (thus violating that particular criterion), hoping that the expanded algebra is still Ramsey but the corresponding space is not. Theorem~\ref{160713a} ensures that this attempt will not succeed.

\section{Open Problems}

Very little has been studied about Ramsey algebras. 
Suppose $f\colon \mathbb{N}^2\times \mathbb{N}^2 \rightarrow \mathbb{N}^2$ is defined by $f((w,x),(y,z))=( w+x,y+z)$.
Then $(\mathbb{N}^2, f)$ is a Ramsey algebra \cite{wcT13c} but $f$ is nowhere associative.\footnote{A binary operation $f$ on $A$ is \emph{nowhere associative} if{f}
$f(f(a,b),c)\neq f(a, f(b,c))$ for all $a,b,c\in A$.}
Hence, the class of Ramsey algebras, even for the case of groupoids (i.e.~with a single binary operation),  appears to be much richer than the class of semigroups.
Furthermore, whether the Cartesian product of two Ramsey groupoids is Ramsey remains elusive.
Therefore, it is fascinating how operations  can mingle together to form a Ramsey algebra, for example, Ramsey algebras of variable words.

Assuming Martin's Axiom, this author \cite{wcT13b} has shown that   every nondegenerate Ramsey algebra has a nonprincipal strongly reductible ultrafilter, analogous to the existence of strongly summable ultrafilters under Martin's Axiom \cite{nH87}. Strongly reductible ultrafilters are necessarily ``idempotent" ultrafilters. 
Carlson asked whether the existence of idempotent ultrafilters for a Ramsey algebra can be proven in $\mathsf{ZFC}$.

Finally, Hindman \cite{nH79a} showed that no ultrafilter on $\mathbb{N}$ is idempotent for addition and multiplication simultaneously. This implies that the integral domain $(\mathbb{Z},+,\cdot)$ does not have idempotent ultrafilters. We ask whether this Hindman's result can be generalized to every infinite integral domain.
A negative answer cannot be deduced from Theorem~\ref{0814} because an algebra with idempotent ultrafilters need not be a Ramsey algebra.

\section*{Acknowledgements}

This paper grows out of the author's thesis submitted as a partial fulfilment for the award of PhD to The Ohio State University under the supervision of Timothy Carlson. 
I owe him a debt of gratitude for introducing the notion of Ramsey algebra. The use of monotone monomials in section \ref{190813} was also suggested by him. Furthermore, the author would like to take this opportunity to express deep appreciation towards his great and supportive mentorship. 
Finally, the author gratefully acknowledge support by an FRGS grant No. 203/PMATHS/6711464 of the Ministry of Education, Malaysia and Universiti Sains Malaysia.


\section*{Conflict of Interests}

The author declares that there is no conflict of interests
regarding the publication of this article.

\end{document}